\documentclass[12pt, reqno, fleqn]{amsart}
\usepackage{amsmath, amsthm, amscd, amsfonts, amssymb, graphicx, color}

\input{mathrsfs.sty}

\newtheorem{theorem}{Theorem}[section]
\newtheorem{lemma}[theorem]{Lemma}

\newtheorem{corollary}[theorem]{Corollary}
\theoremstyle{definition}

\theoremstyle{remark}
\newtheorem{remark}[theorem]{Remark}
\numberwithin{equation}{section}
\begin{document}

\title [Interpolating operator Jensen-type inequalities]{Interpolating operator Jensen-type inequalities for log-convex and superquadratic functions}

\author[M. Bakherad, M. Kian, M. Krni\' c, S.A. Ahmadi]{Mojtaba Bakherad$^1$, Mohsen Kian$^2$, Mario Krni\' c$^3$ and Seyyed Alireza Ahmadi$^4$}

\address{$^1$ Department of Mathematics, Faculty of Mathematics, University of Sistan and Baluchestan, Zahedan, Iran.}
\email{mojtaba.bakherad@yahoo.com; bakherad@member.ams.org}
\address{$^{2}$ Department of Mathematics, Faculty of Basic Sciences, University
of Bojnord, P. O. Box 1339, Bojnord 94531, Iran}
\email{kian@ub.ac.ir and kian@member.ams.org}
\address{$^{3}$ University of Zagreb, Faculty of Electrical Engineering and Computing, Unska 3, 10000 Zagreb, Croatia.}
\email{mario.krnic@fer.hr}

\address{$^4$ Department of Mathematics, Faculty of Mathematics, University of Sistan and Baluchestan, Zahedan, Iran.}
\email{sa.ahmadi@math.usb.ac.ir
}



%~~~~~~~~~~~~~~~~~~~~~~~~~~~~~~~~~~~~~~~~~~~~~~~~~~~~~~~~~~~~~~~~~~~~~~~~~~~~~~~~~~~~~~~~~~~~~~~~~~~~~~~~~~~~~~~~~~~~~~~~~~~~~~~~~~~
\begin{abstract}
Motivated by some recently established operator Jensen-type inequalities related to a usual convexity, in the present paper we derive several more accurate operator Jensen-type inequalities for certain subclasses of convex functions. More precisely, we obtain interpolating series of Jensen-type inequalities for log-convex and non-negative superquadratic functions. In particular, we obtain the corresponding refinements of the Jensen-Mercer operator inequality for such classes of functions.
\end{abstract} \maketitle

%~~~~~~~~~~~~~~~~~~~~~~~~~~~~~~~~~~~~~~~~~~~~~~~~~~~~~~~~~~~~~~~~~~~~~~~~~~~~~~~~~~~~~~~~~~~~~~~~~~~~~~~~~~~~~~~~~~~~~~~~~~~~~~~~~~~
\section{Introduction}
Let ${\mathbb B}({\mathscr H})$ be the $C^*$-algebra of all
bounded linear operators on a complex Hilbert space ${\mathscr H}$ with an identity $I$, and let ${\mathbb B}({\mathscr H})_h$ stands
for the real subspace of ${\mathbb B}({\mathscr H})$ consisting of all self-adjoint operators on $\mathscr H$.
An operator $A\in{\mathbb B}({\mathscr H})_h$ is called positive
if $\langle Ax,x\rangle\geq0$ for all $x\in{\mathscr H }$, and we then write $A\geq0$. In addition, we write $A>0$ if $A$ is
a positive invertible operator.  For
 $A, B\in{\mathbb B}({\mathscr H})_h$ we say
$B\geq A$ if $B-A\geq0$. A map $\Phi$ on ${\mathbb B}({\mathscr H})$ is said to be positive if $\Phi(A)\geq 0$ for each $A\geq 0$ and is called unital if $\Phi(I)=I$.

 The continuous functional calculus is based on the  Gelfand map $f\mapsto f(A)$ which is a
 $*$-isometric isomorphism between the $C^*$-algebra
$C(\rm{sp}(A))$ of all complex-valued continuous functions acting on the spectrum $\rm{sp}(A)$
of a self-adjoint operator $A$ and the $C^*$-algebra generated by $I$ and $A$. The following order preserving property is a consequence of the continuous functional calculus: If $f, g\in C({\rm sp}(A))$, then
$f(t)\geq g(t)$, $t\in{\rm sp}(A)$, implies that $f(A)\geq g(A)$.

 A function $f:J\rightarrow \mathbb{R}$  is convex if
 \begin{eqnarray}\label{definicija_konveksne}f(\alpha x+(1-\alpha)y)\leq \alpha f(x)+ ({1-\alpha})f(y)
 \end{eqnarray} for all $\alpha\in[0,1]$ and all  $x,y\in J$. On the other hand, if $\alpha\not\in[0,1]$ and $x,y\in J$ such that $\alpha x+(1-\alpha)y\in J$, then the last inequality takes form
 \begin{eqnarray}\label{appear}
 \alpha f(x) +({1-\alpha})f(y)\leq f(\alpha x+(1-\alpha)y).
\end{eqnarray}
The relation \eqref{definicija_konveksne} is the most simplest form of the Jensen inequality.

In this article we deal with operator inequalities of the Jensen-type. One of the most famous operator forms of the Jensen inequality is the Davis-Choi-Jensen inequality which is related to operator convexity. Recall that a continuous function $f:J\rightarrow \mathbb{R}$ is operator convex if $f(\alpha A+(1-\alpha)B) \leq \alpha f(A)+(1-\alpha)f(B)$ holds for all $A, B\in {\mathbb B}({\mathscr H})_h$ with spectra in $J$ and all $\alpha\in [0,1]$.
Then, the Davis-Choi-Jensen inequality asserts that if $f:J\rightarrow \mathbb{R}$ is operator convex, then $f(\Phi(A))\leq \Phi(f(A))$ holds for any unital positive linear map $\Phi$ on ${\mathbb B}({\mathscr H})$
and any $A\in {\mathbb B}({\mathscr H})_h$ with spectrum contained in $J$. For some other versions of the Jensen operator inequality related to operator convexity, the reader is referred to monograph \cite{abc} and references therein.

On the other hand, considerable attention is also paid to operator Jensen-type inequalities referring to a mere convexity. Among them, the Jensen-Mercer inequality (see \cite{mat})
asserts that if $f:[m,M]\rightarrow \mathbb{R}$ is a convex function and $\Phi_1, \Phi_2,\ldots , \Phi_n$ are positive linear maps on ${\mathbb B}({\mathscr H})$ with $\sum_{i=1}^n\Phi_i(I)=I$, then
\begin{eqnarray}\label{jenmer}
f\left(M+m-\sum_{i=1}^n\Phi_i(A_i)\right)\leq f(M)+f(m)-\sum_{i=1}^n\Phi_i(f(A_i))
\end{eqnarray}
holds for all $A_1, A_2, \ldots, A_n\in{\mathbb B}({\mathscr H})_h$ with spectra contained in the interval $[m,M]$.

Recently, Moslehian \emph{et.al.} \cite{mos}, showed that if $f:J\rightarrow \mathbb{R}$ is a continuous convex function and  $\Phi$
is a unital positive linear map on ${\mathbb B}({\mathscr H})$, then
 \begin{eqnarray}\label{mos}
f(\Phi(B))+f(\Phi(C))\leq \Phi(f(A))+\Phi (f(D))
\end{eqnarray}
holds for $A, B, C, D\in{\mathbb B}({\mathscr H})_h$ with spectra
contained in $J$ such that $A + D = B + C$ and $A \leq m \leq B, C \leq M \leq D$ for two real numbers $m < M$.
In particular, it has also been showed in \cite{mos} that the Jensen-Mercer inequality \eqref{jenmer} is a consequence of the Jensen-type inequality \eqref{mos}. It should be noticed here that some improvements  of operator relations
\eqref{jenmer} and \eqref{mos} have been established in  \cite{mos2}.

The main objective of the present paper is to derive refinements of operator inequalities \eqref{jenmer} and \eqref{mos} for certain subclasses of convex functions. More precisely, we are going to establish refinements of these inequalities for $\log$-convex and non-negative superquadratic functions. The paper is divided into three sections as follows: after this Introduction, in Section \ref{sec2} we give improved versions of \eqref{jenmer} and \eqref{mos} for $\log$-convex functions, while in Section \ref{sec3} we derive the corresponding results for superquadratic functions. The established Jensen-type relations will be given as interpolating series of inequalities.

%===================================================================================================================================
\section{Interpolating Jensen-type inequalities for log-convex functions}\label{sec2}

 Recall that a  non-negative function $f:J\rightarrow \mathbb{R}$  is called $\log$-convex if $\log f (t)$ is a convex function, that is, if holds the relation
 $f(\alpha x+(1-\alpha)y)\leq f(x)^\alpha f(y)^{1-\alpha}$,
for every $x,y\in J$ and $\alpha\in[0,1]$. Our main goal in this section is to refine operator inequality \eqref{mos}  for $\log$-convex functions.
 To do this, we will first establish several auxiliary scalar relations for such a class of functions.

It is easy to see that every $\log$-convex function $f:J\rightarrow \mathbb{R}$ is also convex. Namely, by virtue of the Young inequality we have
\begin{eqnarray}\label{log}
f(\alpha x+(1-\alpha)y)\leq f(x)^\alpha f(y)^{1-\alpha}\leq\alpha f(x)+(1-\alpha) f(y),
\end{eqnarray}
for $\alpha\in[0,1]$ and $x,y\in J$, which implies the usual convexity. Moreover, if $\alpha\not\in[0,1]$  and  $\alpha x+(1-\alpha)y\in J$, then the signs of inequalities in \eqref{log} are reversed, that is, we have
\begin{align}\label{log1}
f(\alpha x+(1-\alpha)y)\geq f(x)^\alpha f(y)^{1-\alpha}\geq\alpha f(x)+(1-\alpha) f(y).
\end{align}
Clearly, the first inequality in \eqref{log1} follows from convexity of $\log f$  and inequality \eqref{appear}. Moreover, the second inequality in \eqref{log1} follows from reverse of the Young inequality which asserts that
$\alpha a+(1-\alpha) b\leq a^\alpha b^{1-\alpha}$ holds for positive numbers $a$, $b$ and for real number $\alpha$ not belonging to the unit interval $[0,1]$ (for more details, see \cite{bakh}).

Now, we give improved forms of  left inequalities in \eqref{log} and \eqref{log1}.
\begin{lemma}\label{lemica}
Let $f:J\rightarrow \mathbb{R}$ be  $\log$-convex function, let  $x,y\in J$, and let $r(\alpha)=\min\{\alpha,1-\alpha\}$,  where $\alpha\in \mathbb{R}$.
 If $\alpha\in[0,1]$ and $K_f(x,y)=\frac{f^2\left(\frac{x+y}{2}\right)}{f(x)f(y)}$, then
\begin{align}\label{prva}
 f(\alpha x+(1-\alpha)y)\leq K_f(x,y)^{r(\alpha)} f(x)^\alpha f(y)^{1-\alpha}\leq \alpha f(x)+(1-\alpha)f(y).
 \end{align}
 In addition, if $\alpha\not\in[0,1]$  and $\alpha x+(1-\alpha)y\in J$, then the inequality signs in \eqref{prva} are reversed.
 \end{lemma}
\begin{proof}
We first show the left inequality in \eqref{prva}. To do this, we consider two cases depending on weather $\alpha\in[0,\frac{1}{2}]$ or $\alpha\in[\frac{1}{2},1]$.
If $\alpha\in[0,\frac{1}{2}]$, then
\begin{align}\label{alpha}
 f(\alpha x+(1-\alpha)y)\nonumber&
 =f\left(2\alpha \left(\frac{x+y}{2}\right)+(1-2\alpha)y\right)\nonumber\\&
 \leq f\left(\frac{x+y}{2}\right)^{2\alpha}f(y)^{(1-2\alpha)}\nonumber\quad\textrm{(by inequality \eqref{log})}\\&=
 K_f(x,y)^\alpha f(x)^\alpha f(y)^{1-\alpha}.
\end{align}
Further, if $\alpha\in[\frac{1}{2},1]$, then $1-\alpha\in[0,\frac{1}{2}]$, so that
\begin{align}\label{alpha1}
 f(\alpha x+(1-\alpha)y)\nonumber&
 =f\left( (1-2(1-\alpha))x+2(1-\alpha)\left(\frac{x+y}{2}\right)\right)\nonumber\\&
 \leq f(x)^{1-2(1-\alpha)}f\left(\frac{x+y}{2}\right)^{2(1-\alpha)}\quad\textrm{(by inequality \eqref{log})}\nonumber\\&=
 K_f(x,y)^{1-\alpha} f(x)^\alpha f(y)^{1-\alpha}.
 \end{align}
 Therefore, by virtue of   \eqref{alpha} and
 \eqref{alpha1} we have
\begin{align*}
 f(\alpha x+(1-\alpha)y\leq
 K_f(x,y)^{r(\alpha)} f(x)^\alpha f(y)^{1-\alpha},
\end{align*}
as required. The right inequality in \eqref{prva} holds by $\log$-convexity of $f$. Namely,
since $f^2\left(\frac{x+y}{2}\right)\leq{f(x)f(y)}$, it follows that $K_f(x,y)^{r(\alpha)}\leq 1$, so the right inequality in \eqref{prva} is weaker than the corresponding inequality in \eqref{log}.

It remains to prove the corresponding relations with  reversed signs of inequalities. Let $\alpha\not\in[0,1]$ such that $\alpha x+(1-\alpha)y\in J$.  Following the lines as in \eqref{alpha} and \eqref{alpha1}, and utilizing
\eqref{log1} instead of \eqref{log}, we have
  \begin{align*}
 f(\alpha x+(1-\alpha)y)\geq  K_f(x,y)^\alpha f(x)^\alpha f(y)^{1-\alpha}, \quad \alpha<0,
\end{align*}
 and
\begin{align*}
 f(\alpha x+(1-\alpha)y)\geq  K_f(x,y)^{1-\alpha} f(x)^\alpha f(y)^{1-\alpha}, \quad \alpha>1,
\end{align*}
 that is, we have
\begin{align*}
 f(\alpha x+(1-\alpha)y)\geq  K_f(x,y)^{r(\alpha)} f(x)^\alpha f(y)^{1-\alpha}.
\end{align*}
Finally, the second inequality in \eqref{alpha} with the reversed sign of inequality holds do to the fact that $ K_f(x,y)^{r(\alpha)}\geq 1$ for $r(\alpha)<0$.
\end{proof}

\begin{remark}
It should be noticed here that the first inequality in \eqref{prva} can be established as a consequence of superadditivity of the Jensen functional. For more details, the reader is referred to \cite{analele}.
In addition, some results related to Lemma \ref{lemica}  can be found in \cite{sab1}.
\end{remark}

Now,  due to Lemma \ref{lemica}, we
 are ready to prove the first result in this section.

\begin{theorem}\label{main111}
Let $f:J\rightarrow \mathbb{R}$ be  continuous $\log$-convex function and  let $A, B, C, D\in{\mathbb B}({\mathscr H})_h$ be operators with spectra
contained in $J$ such  that $A \leq m \leq B, C \leq M \leq D$ for  real numbers $m \leq M$.
If one of the following conditions
\begin{align*}
&({\rm i})\,\, B + C \leq A +D\quad \textrm{and}\quad f(m)\leq f (M)
\\&({\rm ii})\, A + D \leq B +C \quad\textrm{and}\quad f(M) \leq f (m)
\end{align*}
is satisfied, then
\begin{align*}
 f(B)&+f(C)\\&\leq K_f(m,M)^{\widetilde{B}}f(m)^{\frac{M-B}{M-m}}f(M)^{\frac{B-m}{M-m}}+K_f(m,M)^{\widetilde{C}}f(m)^{\frac{M-C}{M-m}}f(M)^{\frac{C-m}{M-m}}
 \\&\leq \frac{2M -B-C}{M-m}f(m)+{\frac{B+C-2m}{M-m}}f(M)
   \\
 &\leq K_f(m,M)^{\widetilde{A}}f(m)^{\frac{M-A}{M-m}}f(M)^{\frac{A-m}{M-m}}+K_f(m,M)^{\widetilde{D}}f(m)^{\frac{M-D}{M-m}}f(M)^{\frac{D-m}{M-m}}
 \\&\leq f(A)+f(D),
\end{align*}
where $\widetilde{t}=\frac{1}{2}-\frac{1}{M-m}\left|t-\frac{m+M}{2}\right|$ and $K_f(m,M)$ is defined in Lemma \ref{lemica}.
\end{theorem}
\begin{proof}
First, it should be noticed here that  $$r\left(\frac{M-t}{M-m}\right)=\min\left\{\frac{M-t}{M-m},\frac{t-m}{M-m}\right\}=\frac{1}{2}-\frac{1}{M-m}\left|t-\frac{m+M}{2}\right|=\widetilde{t}.$$
Further, by virtue of  Lemma \ref{lemica}, it follows that
\begin{align}\label{help0}
 f(t)&=f\left(\frac{M-t}{M-m}m+\frac{t-m}{M-m}M\right)\nonumber\\&\leq K_f(m,M)^{\widetilde{t}}f(m)^{\frac{M-t}{M-m}}f(M)^{\frac{t-m}{M-m}}\nonumber\\
 &\leq \frac{M-t}{M-m}f(m)+\frac{t-m}{M-m}f(M)\nonumber \\
 &= \frac{f(M)-f(m)}{M-m}t+\frac{f(m)M-f(M)m}{M-m}
\end{align}
holds for $t\in[m,M]$,
while for $t\in J\setminus[m,M]$ the inequality signs in \eqref{help0} are reversed.
Now, since  $A \leq m \leq B, C \leq M \leq D$, applying the functional calculus to the above scalar relations yields
\begin{align}\label{help1}
 f(B)&\leq  K_f(m,M)^{\widetilde{B}}f(m)^{\frac{M-B}{M-m}}f(M)^{\frac{B-m}{M-m}}\nonumber\\
 &\leq \frac{M-B}{M-m}f(m)+\frac{B-m}{M-m}f(M)\nonumber \\
 &= \frac{f(M)-f(m)}{M-m}B+\frac{f(m)M-f(M)m}{M-m},
\end{align}
\begin{align}\label{help2}
 f(C)&\leq  K_f(m,M)^{\widetilde{C}}f(m)^{\frac{M-C}{M-m}}f(M)^{\frac{C-m}{M-m}}\nonumber\\
 &\leq \frac{M-C}{M-m}f(m)+\frac{C-m}{M-m}f(M)\nonumber \\
 &= \frac{f(M)-f(m)}{M-m}C+\frac{f(m)M-f(M)m}{M-m},
\end{align}
\begin{align*}
 f(A)&\geq  K_f(m,M)^{\widetilde{A}}f(m)^{\frac{M-A}{M-m}}f(M)^{\frac{A-m}{M-m}}\nonumber\\&\geq \frac{f(M)-f(m)}{M-m}A+\frac{f(m)M-f(M)m}{M-m},
\end{align*}
and
\begin{align*}
 f(D)&\geq K_f(m,M)^{\widetilde{D}}f(m)^{\frac{M-D}{M-m}}f(M)^{\frac{D-m}{M-m}}\nonumber\\&\geq \frac{f(M)-f(m)}{M-m}D+\frac{f(m)M-f(M)m}{M-m}.
\end{align*}
Therefore, if one of  conditions (i) or (ii) in the statement of Theorem is fulfilled, we have
\begin{align*}
 &f(B)+f(C)\\&\leq K_f(m,M)^{\widetilde{B}}f(m)^{\frac{M-B}{M-m}}f(M)^{\frac{B-m}{M-m}}+K_f(m,M)^{\widetilde{C}}f(m)^{\frac{M-C}{M-m}}f(M)^{\frac{C-m}{M-m}}
 \\
 &\leq \frac{2M -B-C}{M-m}f(m)+{\frac{B+C-2m}{M-m}}f(M)
   \\
 &= \frac{f(M)-f(m)}{M-m}(B+C)+2\frac{f(m)M-f(M)m}{M-m}
 \\&\leq\frac{f(M)-f(m)}{M-m}(A+D)+2\frac{f(m)M-f(M)m}{M-m}
 \\&\leq K_f(m,M)^{\widetilde{A}}f(m)^{\frac{M-A}{M-m}}f(M)^{\frac{A-m}{M-m}}+K_f(m,M)^{\widetilde{D}}f(m)^{\frac{M-D}{M-m}}f(M)^{\frac{D-m}{M-m}}
 \\&\leq f(A)+f(D),
\end{align*}
which completes the proof.
\end{proof}
Our first application of  Theorem \ref{main111} refers to a power function $f:\mathbb{R}_+\rightarrow \mathbb{R}$ defined by $f(x)=x^{p}$, $p\leq 0$.
Obviously, $f$ is $\log$-convex and we have the following consequence.

\begin{corollary}\label{main1111}
Let $A, B, C, D\in{\mathbb B}({\mathscr H})_h$ be positive invertible  such  that $0<A \leq m \leq B, C \leq M \leq D$ for  non-negative real numbers $m \leq M$.
If $A +D\leq B + C$ and $p\leq0$,
 then
\begin{align*}
 B^p&+C^p\\&\leq\Big( \frac{m+M}{2\sqrt{mM}} \Big)^{2p\widetilde{B}}m^{p(\frac{M-B}{M-m})}M^{p(\frac{B-m}{M-m})}+\Big( \frac{m+M}{2\sqrt{mM}} \Big)^{2p\widetilde{C}}m^{p(\frac{M-C}{M-m})}M^{p(\frac{C-m}{M-m})}
 \\
 &\leq \frac{2M -B-C}{M-m}m^{p}+{\frac{B+C-2m}{M-m}}M^{p}
   \\
 &\leq \Big( \frac{m+M}{2\sqrt{mM}} \Big)^{2p\widetilde{A}}m^{p(\frac{M-A}{M-m})}M^{p(\frac{A-m}{M-m})}+\Big( \frac{m+M}{2\sqrt{mM}} \Big)^{2p\widetilde{D}}m^{p(\frac{M-D}{M-m})}M^{p(\frac{D-m}{M-m})}
 \\&\leq A^p+D^p,
\end{align*}
where $\widetilde{t}$ is defined in the statement of Theorem \ref{main111}.
\begin{remark}\label{midpoint1}
Let $m$ and $M$ be real numbers such that $m<M$. In \cite{mos}, the authors defined the subset $\Omega$ of ${\mathbb B}({\mathscr H})_h\times {\mathbb B}({\mathscr H})_h$ by $\Omega =\left\{ (A,B); A \leq m \leq \frac{A+B}{2} \leq M \leq B \right\}$. Now, if $f:J\rightarrow \mathbb{R}$ is a  continuous $\log$-convex function and $(A,D)\in\Omega$ with spectra contained in $J$, we obtain the Jensen-type relation
\begin{align*}
&f\Big(\frac{A+D}{2}\Big) \\&\quad\leq K_f(m,M)^{\widetilde{\frac{A+D}{2}}}f(m)^{\frac{2M-A-D}{2(M-m)}}f(M)^{\frac{A+D-2m}{2(M-m)}}
 \\&\quad\leq \frac{2M -A-D}{2(M-m)}f(m)+{\frac{A+D-2m}{2(M-m)}}f(M)
   \\
 &\quad\leq \frac{1}{2}K_f(m,M)^{\widetilde{A}}f(m)^{\frac{M-A}{M-m}}f(M)^{\frac{A-m}{M-m}}+\frac{1}{2}K_f(m,M)^{\widetilde{D}}f(m)^{\frac{M-D}{M-m}}f(M)^{\frac{D-m}{M-m}}
 \\&\quad\leq \frac{1}{2}f(A)+\frac{1}{2}f(D).
\end{align*}
Clearly, this follows by putting $B=C=\frac{A+D}{2}$ in Theorem \ref{main111}.
\end{remark}

\end{corollary}
Our next intention is to establish a variant of Theorem \ref{main111} containing unital positive linear mappings. The following result provides the interpolating series of inequalities for  relation \eqref{mos},
established in \cite{mos}.

\begin{theorem}\label{main11111}
Let $f:J\rightarrow \mathbb{R}$ be  continuous $\log$-convex function and let $A, B, C, D\in{\mathbb B}({\mathscr H})_h$ be  operators with spectra
contained in $J$ such that $A + D = B + C$ and $A \leq m \leq B, C \leq M \leq D$ for  real numbers $m < M$. If $\Phi$
is a unital positive linear map on ${\mathbb B}({\mathscr H})$, then
\begin{align*}
 \Phi(f(B))&+\Phi(f(C))\\&\leq  \Phi\left(K_f(m,M)^{\widetilde{B}}f(m)^{\frac{M-B}{M-m}}f(M)^{\frac{B-m}{M-m}}\right)
 \\&\quad+\Phi\left(K_f(m,M)^{\widetilde{C}}f(m)^{\frac{M-C}{M-m}}f(M)^{\frac{C-m}{M-m}}\right)\\&\leq \frac{2M -\Phi(B+C)}{M-m}f(m)+{\frac{\Phi(B+C)-2m}{M-m}}f(M)
   \\&\leq K_f(m,M)^{\widetilde{\Phi(A)}}f(m)^{\frac{M-\Phi(A)}{M-m}}f(M)^{\frac{\Phi(A)-m}{M-m}}
 \\&\quad+K_f(m,M)^{\widetilde{\Phi(D)}}f(m)^{\frac{M-\Phi(D)}{M-m}}f(M)^{\frac{\Phi(D)-m}{M-m}}
 \\&\leq f(\Phi(A))+f(\Phi(D)),
\end{align*}
where $\widetilde{t}=\frac{1}{2}-\frac{1}{M-m}\left|t-\frac{m+M}{2}\right|$ and $K_f(m,M)$ is defined in Lemma \ref{lemica}.
\end{theorem}
\begin{proof}
Applying the positive linear map $\Phi$ to relations \eqref{help1} and \eqref{help2},  we have
\begin{align*}
 \Phi(f(B))&\leq  \Phi\left(K_f(m,M)^{\widetilde{B}}f(m)^{\frac{M-B}{M-m}}f(M)^{\frac{B-m}{M-m}}\right)\nonumber\\&\leq \frac{M-\Phi(B)}{M-m}f(m)+{\frac{\Phi(B)-m}{M-m}}f(M)
\end{align*}
and
\begin{align*}
 \Phi(f(C))&\leq  \Phi\left(K_f(m,M)^{\widetilde{C}}f(m)^{\frac{M-C}{M-m}}f(M)^{\frac{C-m}{M-m}}\right)\nonumber\\&\leq \frac{M-\Phi(C)}{M-m}f(m)+{\frac{\Phi(C)-m}{M-m}}f(M).
\end{align*}
Moreover, since $\Phi(A)\leq m$ and $M\leq \Phi(D)$, utilizing the scalar relation \eqref{help0} with reversed signs of inequalities, we have
\begin{align*}
 f(\Phi(A))&\geq  K_f(m,M)^{\widetilde{\Phi(A)}}f(m)^{\frac{M-\Phi(A)}{M-m}}f(M)^{\frac{\Phi(A)-m}{M-m}}\nonumber\\&\geq \frac{M-\Phi(A)}{M-m}f(m)+{\frac{\Phi(A)-m}{M-m}}f(M)
\end{align*}
and
\begin{align*}
 f(\Phi(D))&\geq K_f(m,M)^{\widetilde{\Phi(D)}}f(m)^{\frac{M-\Phi(D)}{M-m}}f(M)^{\frac{\Phi(D)-m}{M-m}}\nonumber\\&\geq \frac{M-\Phi(D)}{M-m}f(m)+{\frac{\Phi(D)-m}{M-m}}f(M).
\end{align*}
Now, adding relations for $\Phi(f(B))$ and $\Phi(f(C))$,  taking into account that $B+C=A+D$, and utilizing estimates for $f(\Phi(A))$ and $f(\Phi(D))$, we have
\begin{align*}
 \Phi(f(B))&+\Phi(f(C))\\&\leq  \Phi\left(K_f(m,M)^{\widetilde{B}}f(m)^{\frac{M-B}{M-m}}f(M)^{\frac{B-m}{M-m}}\right)
 \\&\quad+\Phi\left(K_f(m,M)^{\widetilde{C}}f(m)^{\frac{M-C}{M-m}}f(M)^{\frac{C-m}{M-m}}\right)\\&\leq \frac{2M -\Phi(B+C)}{M-m}f(m)+{\frac{\Phi(B+C)-2m}{M-m}}f(M)
 \\&=\frac{2M-\Phi(A+D)}{M-m}f(m)+{\frac{\Phi(A+D)-2m}{M-m}}f(M)
 \\&\leq\frac{M-\Phi(A)}{M-m}f(m)+{\frac{\Phi(A)-m}{M-m}}f(M)
 \\&\quad+\frac{M-\Phi(D)}{M-m}f(m)+{\frac{\Phi(D)-m}{M-m}}f(M)
 \\&\leq K_f(m,M)^{\widetilde{\Phi(A)}}f(m)^{\frac{M-\Phi(A)}{M-m}}f(M)^{\frac{\Phi(A)-m}{M-m}}
 \\&\quad+K_f(m,M)^{\widetilde{\Phi(D)}}f(m)^{\frac{M-\Phi(D)}{M-m}}f(M)^{\frac{\Phi(D)-m}{M-m}}
 \\&\leq f(\Phi(A))+f(\Phi(D)),
\end{align*}
which completes the proof.
\end{proof}
\begin{remark}\label{slicno}
Suppose that  the assumptions of Theorem \ref{main11111} are fulfilled. Then, according to the same argument as in the proof of Theorem \ref{main11111}, we easily get the following interpolating series of inequalities:
\begin{align*}
 f(\Phi(B))&+f(\Phi(C))\\&\leq  K_f(m,M)^{\widetilde{\Phi(B)}}f(m)^{\frac{M-\Phi(B)}{M-m}}f(M)^{\frac{\Phi(B)-m}{M-m}}
 \\&\quad+K_f(m,M)^{\widetilde{\Phi(C)}}f(m)^{\frac{M-\Phi(C)}{M-m}}f(M)^{\frac{\Phi(C)-m}{M-m}}\\&\leq \frac{2M -\Phi(B+C)}{M-m}f(m)+{\frac{\Phi(B+C)-2m}{M-m}}f(M)
   \\&\leq\Phi\left(K_f(m,M)^{\widetilde{A}}f(m)^{\frac{M-A}{M-m}}f(M)^{\frac{A-m}{M-m}}\right)
 \\&\quad+\Phi\left(K_f(m,M)^{\widetilde{D}}f(m)^{\frac{M-D}{M-m}}f(M)^{\frac{D-m}{M-m}}\right)
      \\&\leq \Phi(f(A))+\Phi(f(D))
\end{align*}
and
\begin{align*}
 \Phi(f(B))&+f(\Phi(C))\\&\leq  \Phi\left(K_f(m,M)^{\widetilde{B}}f(m)^{\frac{M-B}{M-m}}f(M)^{\frac{B-m}{M-m}}\right)
 \\&\quad+K_f(m,M)^{\widetilde{\Phi(C)}}f(m)^{\frac{M-\Phi(C)}{M-m}}f(M)^{\frac{\Phi(C)-m}{M-m}}\\&\leq \frac{2M -\Phi(B+C)}{M-m}f(m)+{\frac{\Phi(B+C)-2m}{M-m}}f(M)
    \\&\leq K_f(m,M)^{\widetilde{\Phi(A)}}f(m)^{\frac{M-\Phi(A)}{M-m}}f(M)^{\frac{\Phi(A)-m}{M-m}}
 \\&\quad+\Phi\left(K_f(m,M)^{\widetilde{D}}f(m)^{\frac{M-D}{M-m}}f(M)^{\frac{D-m}{M-m}}\right)
  \\&\leq f(\Phi(A))+\Phi(f(D)).
\end{align*}
\end{remark}
Our next intention is to establish interpolating series for the Jensen-Mercer inequality  \eqref{jenmer}. To do this, we give  a multidimensional version of the interpolating series from the previous remark.

\begin{corollary}\label{main21}
Let $f:J\rightarrow \mathbb{R}$ be  continuous $\log$-convex function. Let $A_i, B_i, C_i, D_i\in {\mathbb B}({\mathscr H})_h$ be  operators with spectra contained in $J$ such that $A_i + D_i = B_i + C_i$ and $A_i \leq m \leq B_i, C_i \leq M \leq D_i$, $i=1,2,\ldots ,n$, for  real numbers $m < M$.  If $\Phi_1,\Phi_2,\cdots,\Phi_n$ are positive linear maps on ${\mathbb B}({\mathscr H})$ satisfying
$\sum_{i=1}^{n}\Phi_i(I)=I$, then
\begin{align*}
 \sum_{i=1}^{n}&\Phi_i(f(B_i))+f\left(\sum_{i=1}^{n}\Phi_i(C_i)\right)\nonumber
 \\&\leq  \sum_{i=1}^{n}\Phi_i\left(K_f(m,M)^{\widetilde{B_i}}f(m)^{\frac{M-B_i}{M-m}}f(M)^{\frac{B_i-m}{M-m}}\right)\nonumber
 \\&\quad+K_f(m,M)^{\widetilde{\sum_{i=1}^{n}\Phi_i(C_i)}}f(m)^{\frac{M-\sum_{i=1}^{n}\Phi_i(C_i)}{M-m}}f(M)^{\frac{\sum_{i=1}^{n}\Phi_i(C_i)-m}{M-m}}\nonumber\\&\leq \frac{2M -\sum_{i=1}^{n}\Phi_i(B_i+C_i)}{M-m}f(m)+{\frac{\sum_{i=1}^{n}\Phi_i(B_i+C_i)-2m}{M-m}}f(M)\nonumber
    \\&\leq K_f(m,M)^{\widetilde{\sum_{i=1}^{n}\Phi_i(A_i)}}f(m)^{\frac{M-\sum_{i=1}^{n}\Phi_i(A_i)}{M-m}}f(M)^{\frac{\sum_{i=1}^{n}\Phi_i(A_i)-m}{M-m}}\nonumber
 \\&\quad+\sum_{i=1}^{n}\Phi_i\left(K_f(m,M)^{\widetilde{D_i}}f(m)^{\frac{M-D_i}{M-m}}f(M)^{\frac{D_i-m}{M-m}}\right)\nonumber
  \\&\leq f\left(\sum_{i=1}^{n}\Phi_i(A_i)\right)+\sum_{i=1}^{n}\Phi_i(f(D_i)).
\end{align*}
\end{corollary}

\begin{remark}
Suppose that  $B_i\in {\mathbb B}({\mathscr H})_h$, $i=1,2,\ldots , n$, are   operators with spectra contained in $[m,M]$, and let
$\Phi_i$, $i=1,2,\ldots , n$,  be
positive linear maps on ${\mathbb B}({\mathscr H})$ such that $\sum_{i=1}^n\Phi_i(I)=I$. Then, Corollary \ref{main21} provides interpolating series of the Jensen-Mercer inequality for
 continuous  $\log$-convex function on $[m,M]$. To see this, we first define $C_i=M+m-B_i$ $i=1,2,\ldots , n$. Clearly, $m\leq C_i\leq M$ and $B_i+C_i=M+m$. Now,
  setting  $A_i=mI$ and $D_i=MI$, $i=1,2,\ldots , n$, Corollary \ref{main21} reduces to
{\small\begin{align*}
 &\sum_{i=1}^{n}\Phi_i(f(B_i))+f\left(M+m-\sum_{i=1}^{n}\Phi_i(B_i)\right)
 \\&\leq  \sum_{i=1}^{n}\Phi_i\left(K_f(m,M)^{\widetilde{B_i}}f(m)^{\frac{M-B_i}{M-m}}f(M)^{\frac{B_i-m}{M-m}}\right)
 \\&\quad+K_f(m,M)^{\widetilde{M+m-\sum_{i=1}^{n}\Phi_i(B_i)}}f(m)^{\frac{m-\sum_{i=1}^{n}\Phi_i(B_i)}{M-m}}f(M)^{\frac{M-\sum_{i=1}^{n}\Phi_i(B_i)}{M-m}}\\&\leq f\left(m\right)+f(M),
\end{align*}}
since $\sum_{i=1}^{n}\Phi_i(B_i+C_i)=\sum_{i=1}^{n}\Phi_i((m+M)I)=m+M$ and $\widetilde{m}=\widetilde{M}=0$. Clearly, this yields a term that interpolates between the left-hand side and the right-hand side of the Jensen-Mercer inequality \eqref{jenmer}.
\end{remark}

\section{Employing superquadraticity}\label{sec3}

Our goal in this section is to derive analogues of the results from the previous section for superquadratic functions.

For the reader's convenience, we  recall the definition  and several basic properties of superquadratic functions, introduced by Abramovich \emph{et.al.} \cite{AJS}
(for more details, see also  \cite{AJS1}). A function $f: [0,\infty)\rightarrow \mathbb{R}$  is called superquadratic provided that for each
 $s
\geq0$ there exists a
constant $c_s\in \mathbb{R}$ such that
\begin{align}\label{super}
f (t)-f (s)- f (|t-s|)\geq   c_s(t - s)
\end{align}
holds for all $t\geq0$. Superquadratic functions are closely connected to convex
functions. In particular, at the first glance, condition (\ref{super})
seems to be stronger than convexity. However, if $f$ takes negative
values, it may be considerably weaker. Just to see how poorly
behaved superquadratic functions can be, we note that any
function $f$ with values in the closed interval $[-2,-1]$  is
superquadratic, since in this case the left-hand side of (\ref{super}) is non-negative, so we can put $c_s=0$.
On the other hand, non-negative superquadratic
functions are much better behaved, namely, they are convex.

A common example of a superquadratic function is a power function. Namely, it has been showed in \cite{AJS} that a function
$f: [0,\infty)\rightarrow \mathbb{R}$ defined by $f(x)=x^{p}$ is superquadratic for $p\geq 2$.

One of the most useful characterizations of superquadratic functions is an extension of the Jensen inequality, given here in the most simplest form.
A function $f:[0,\infty)\rightarrow \mathbb{R}$  is superquadratic if and only if
\begin{align}\label{super2}
f(\alpha x+(1-\alpha)y)&\leq \alpha f(x)+(1-\alpha) f(y)\nonumber\\&\quad-\alpha f((1-\alpha)|x-y|)-(1-\alpha)f(\alpha|x-y|)
\end{align}
for every $x,y\in J$ and $\alpha\in[0,1]$. This relation will be the crucial point in this section.
For more generalized form of \eqref{super2} the reader is refered to \cite{AJS}.

Now, we are able to give the main result in this section, that is, we will establish an analogue of Theorem \ref{main11111} for superquadratic functions.
In the case of a non-negative superquadratic function, the following theorem provides  a refinement of  the Jensen-type inequality \eqref{mos}.

\begin{theorem}\label{main1}
Let $f:[0,\infty)\rightarrow \mathbb{R}$ be  continuous superquadratic function and let $A, B, C, D\in{\mathbb B}({\mathscr H})_h$ be  such that $A + D = B + C$ and $0\leq A \leq m \leq B, C \leq M \leq D$, for  non-negative real numbers $0\leq m < M$. If $\Phi$
is a unital positive linear map on ${\mathbb B}({\mathscr H})$, then
{\small\begin{align}\label{so1}
&f(\Phi(B))+f(\Phi(C))\nonumber\\&\leq
\Phi(f(A))+\Phi(f(D))-\frac{M-\Phi(B)}{M-m}f(\Phi(B)-m)-{\frac{\Phi(B)-m}{M-m}}f(M-\Phi(B))\nonumber\\&\quad
 -\frac{M-\Phi(C)}{M-m}f(\Phi(C)-m)-{\frac{\Phi(C)-m}{M-m}}f(M-\Phi(C))-\Phi(f(m-A))\nonumber
 \\&\quad-{\frac{m-\Phi(A)}{M-m}}f(M-m)-\Phi(f(D-M))-{\frac{\Phi(D)-M}{M-m}}f(M-m).
\end{align}}
\end{theorem}
\begin{proof}
We utilize characterization \eqref{super2} of a superquadratic function. Namely, if $t\in[m,M]$, then considering \eqref{super2} with $x=m$, $y=M$ and $\alpha=\frac{M-t}{M-m}$ provides
\begin{align*}
f(t)\leq \frac{M-t}{M-m}f(m)+{\frac{t-m}{M-m}}f(M)-\frac{M-t}{M-m}f(t-m)-{\frac{t-m}{M-m}}f(M-t).
\end{align*}
If $t<m$, then $m\in [t, M]$ so applying \eqref{super2} with $x=t$, $y=M$ and $\alpha=\frac{M-m}{M-t}$, it follows that
\begin{align*}
f(t)\geq \frac{M-t}{M-m}f(m)+{\frac{t-m}{M-m}}f(M)+f(m-t)+{\frac{m-t}{M-m}}f(M-m).
\end{align*}
Finally, if $t>M$, then $M\in[m,t]$, so putting $x=m$, $y=t$ and $\alpha=\frac{t-M}{t-m}$ in \eqref{super2} yields
\begin{align*}
f(t)\geq \frac{M-t}{M-m}f(m)+{\frac{t-m}{M-m}}f(M)+f(t-M)+{\frac{t-M}{M-m}}f(M-m).
\end{align*}
Now, since $m \leq \Phi(B), \Phi(C) \leq M$, applying functional calculus to the first scalar inequality, we have
\begin{align*}
f(\Phi(B))&\leq \frac{M-\Phi(B)}{M-m}f(m)+{\frac{\Phi(B)-m}{M-m}}f(M)\nonumber
\\&\quad-\frac{M-\Phi(B)}{M-m}f(\Phi(B)-m)-{\frac{\Phi(B)-m}{M-m}}f(M-\Phi(B))
 \end{align*}
and
\begin{align*}
f(\Phi(C))&\leq \frac{M-\Phi(C)}{M-m}f(m)+{\frac{\Phi(C)-m}{M-m}}f(M)\nonumber
\\&\quad-\frac{M-\Phi(C)}{M-m}f(\Phi(C)-m)-{\frac{\Phi(C)-m}{M-m}}f(M-\Phi(C)).
\end{align*}
Similarly, since $A\leq m$ and $M\leq D$, applying functional calculus to the remaing two scalar inequalities and then, taking into account linearity of $\Phi$, we obtain
\begin{align*}
\Phi(f(A))&\geq \frac{M-\Phi(A)}{M-m}f(m)+{\frac{\Phi(A)-m}{M-m}}f(M)\nonumber\\&\quad+\Phi(f(m-A))+{\frac{m-\Phi(A)}{M-m}}f(M-m)
\end{align*}
and
\begin{align*}
\Phi(f(D))&\geq \frac{M-\Phi(D)}{M-m}f(m)+{\frac{\Phi(D)-m}{M-m}}f(M)\nonumber\\&\quad+\Phi(f(D-M))+{\frac{\Phi(D)-M}{M-m}}f(M-m).
\end{align*}
Now, summing the estimates for $f(\Phi(B))$ and $f(\Phi(C))$, it follows that
\begin{align*}
&f(\Phi(B))+f(\Phi(C))\\&\leq L(\Phi(B))-\frac{M-\Phi(B)}{M-m}f(\Phi(B)-m)-{\frac{\Phi(B)-m}{M-m}}f(M-\Phi(B))\\&\quad+
 L(\Phi(C))-\frac{M-\Phi(C)}{M-m}f(\Phi(C)-m)-{\frac{\Phi(C)-m}{M-m}}f(M-\Phi(C)),
\end{align*}
where $L(t)=\frac{M-t}{M-m}f(m)+{\frac{t-m}{M-m}}f(M)$. In addition, since $A + D = B + C$, the right-hand side of the latter inequality can be rewritten as
\begin{align*}
&\Phi(L(A))+\Phi(L(D))-\frac{M-\Phi(B)}{M-m}f(\Phi(B)-m)-{\frac{\Phi(B)-m}{M-m}}f(M-\Phi(B))\\&\quad
 -\frac{M-\Phi(C)}{M-m}f(\Phi(C)-m)-{\frac{\Phi(C)-m}{M-m}}f(M-\Phi(C)),
\end{align*}due to linearity of $\Phi$ and $L$.
Finally, taking into account estimates for $\Phi(f(A))$ and $\Phi(f(D))$ we get \eqref{so1}, as claimed.
\end{proof}

Our next application of
Theorem \ref{main1} refers to power superquadratic functions. We have already discussed that a power function  $f(t)=t^p$, $t\geq 0$, is superquadratic for $p\geq 2$.
Applying the theorem to such a class of functions, we obtain the following result.

\begin{corollary}
Let $A, B, C, D\in{\mathbb B}({\mathscr H})_h$ be    such  that $0\leq A \leq m \leq B, C \leq M \leq D$, for  non-negative real numbers $m \leq M$,  and let $\Phi$
be a unital positive linear map on ${\mathbb B}({\mathscr H})$.
If $A +D= B + C$ and $p\geq2$,
 then
{\small\begin{align*}
\Phi&(B)^p+\Phi(C)^p\\&\quad+\frac{M-\Phi(B)}{M-m}(\Phi(B)-m)^p+{\frac{\Phi(B)-m}{M-m}}(M-t)^p\\&\quad
 +\frac{M-\Phi(C)}{M-m}(\Phi(C)-m)^p+{\frac{\Phi(C)-m}{M-m}}(M-t)^p\\&\leq
\Phi(A^p)+\Phi(D^p)-\Phi((m-A)^p)-{\frac{m-\Phi(A)}{M-m}}(M-m)^p
 \\&\quad-\Phi((D-M)^p)-{\frac{\Phi(D)-M}{M-m}}(M-m)^p.
\end{align*}}
\end{corollary}
\begin{remark}
Suppose that the  assumptions as in  Theorem \ref{main1} are fulfilled. Then, similarly  to Remark \ref{slicno} we also obtain the following two versions of Theorem \ref{main1}:
{\small\begin{align*}
&\Phi(f(B))+\Phi(f(C))\\&\leq
f(\Phi(A))+f(\Phi(D))-\frac{M-\Phi(B)}{M-m}\Phi(f(B)-m)-{\frac{\Phi(B)-m}{M-m}}\Phi(f(M-B))\\&\quad
 -\frac{M-\Phi(C)}{M-m}\Phi(f(C)-m)-{\frac{\Phi(C)-m}{M-m}}f(M-\Phi(C))-f(m-\Phi(A))
 \\&\quad-{\frac{m-\Phi(A)}{M-m}}f(M-m)-f(\Phi(D)-M)-{\frac{\Phi(D)-M}{M-m}}f(M-m)
\end{align*}}
and
{\small\begin{align}\label{so3}
&f(\Phi(B))+\Phi(f(C))\nonumber\\&\leq
\Phi(f(A))+f(\Phi(D))-\frac{M-\Phi(B)}{M-m}f(\Phi(B)-m)-{\frac{\Phi(B)-m}{M-m}}f(M-\Phi(B))\nonumber\\&\quad
 -\frac{M-\Phi(C)}{M-m}\Phi(f(C)-m)-{\frac{\Phi(C)-m}{M-m}}f(M-\Phi(C))-\Phi(f(m-A))\nonumber\\&\quad-{\frac{m-\Phi(A)}{M-m}}f(M-m)
 -f(\Phi(D)-M)-{\frac{\Phi(D)-M}{M-m}}f(M-m).
\end{align}}
\end{remark}

Now, our goal is  to establish improved version of the Jensen-Mercer inequality \eqref{jenmer} based on superquadraticity. To do this, we first give  multidimensional versions of  inequalities \eqref{so1} and \eqref{so3}.
Namely, following the lines of the proof of Theorem \ref{main1}, we obtain the following result.

\begin{corollary}\label{koro_merc}
Let $f:[0,\infty)\rightarrow \mathbb{R}$ be  continuous superquadratic function. Let $A_i, B_i, C_i, D_i\in {\mathbb B}({\mathscr H})_h$ be  such that $A_i + D_i = B_i + C_i$ and $0\leq A_i \leq m \leq B_i, C_i \leq M \leq D_i$, $i=1,2,\ldots ,n$, for  non-negative real numbers $m < M$.  If $\Phi_1,\Phi_2,\cdots,\Phi_n$ are positive linear maps on ${\mathbb B}({\mathscr H})$ satisfying
$\sum_{i=1}^{n}\Phi_i(I)=I$, then
{\footnotesize\begin{align*}
&f\left(\sum_{i=1}^{n}\Phi_i(B_i)\right)+f\left(\sum_{i=1}^{n}\Phi_i(C_i)\right)\\&\quad
+\frac{M-\sum_{i=1}^{n}\Phi_i(B_i)}{M-m}f\left(\sum_{i=1}^{n}\Phi_i(B_i)-m\right)
+{\frac{\sum_{i=1}^{n}\Phi_i(B_i)-m}{M-m}}f\left(M-\sum_{i=1}^{n}\Phi_i(B_i)\right)\\&\quad
 +\frac{M-\sum_{i=1}^{n}\Phi_i(C_i)}{M-m}f\left(\sum_{i=1}^{n}\Phi_i(C_i)-m\right)+
 {\frac{\sum_{i=1}^{n}\Phi_i(C_i)-m}{M-m}}f\left(M-\sum_{i=1}^{n}\Phi_i(C_i)\right)
\\&\leq
\sum_{i=1}^{n}\Phi_i(f(A_i))+\sum_{i=1}^{n}\Phi_i(f(D_i))
 -\sum_{i=1}^{n}\Phi_i(f(m-A_i))-{\frac{m-\sum_{i=1}^{n}\Phi_i(A_i)}{M-m}}f(M-m)
 \\&\quad-\sum_{i=1}^{n}\Phi_i(f(D_i-M))-{\frac{\sum_{i=1}^{n}\Phi_i(D_i)-M}{M-m}}f(M-m)
\end{align*}}
and
{\footnotesize\begin{align*}
&\sum_{i=1}^{n}\Phi_i(f\left(B_i)\right)+f\left(\sum_{i=1}^{n}\Phi_i(C_i)\right)\nonumber\\&\quad+
\frac{M-\sum_{i=1}^{n}\Phi_i(B_i)}{M-m}\left(\sum_{i=1}^{n}\Phi_i(f(B_i))-m\right)
+{\frac{\sum_{i=1}^{n}\Phi_i(B_i)-m}{M-m}}\left(M-\sum_{i=1}^{n}\Phi_i(f(B_i))\right)\nonumber\\&\quad
 +\frac{M-\sum_{i=1}^{n}\Phi_i(C_i)}{M-m}f\left(\sum_{i=1}^{n}\Phi_i(C_i)-m\right)
 +{\frac{\sum_{i=1}^{n}\Phi_i(C_i)-m}{M-m}}f\left(M-\sum_{i=1}^{n}\Phi_i(C_i)\right)\nonumber\\&\leq
f\left(\sum_{i=1}^{n}\Phi_i(A_i)\right)+\sum_{i=1}^{n}\Phi_i(f(D_i))
-f\left(m-\sum_{i=1}^{n}\Phi_i(A_i)\right)\nonumber\\&\quad-{\frac{m-\sum_{i=1}^{n}\Phi_i(A_i)}{M-m}}f(M-m)
  -\sum_{i=1}^{n}\Phi_i(f\left(D_i-M)\right)-{\frac{\sum_{i=1}^{n}\Phi_i(D_i)-M}{M-m}}f(M-m).
\end{align*}}
\end{corollary}
\begin{remark}
Suppose that  $B_i\in {\mathbb B}({\mathscr H})_h$, $i=1,2,\ldots , n$, are  positive operators with spectra contained in $[m,M]\subseteq [0,\infty)$, and let
$\Phi_i$, $i=1,2,\ldots , n$,  be
positive linear maps on ${\mathbb B}({\mathscr H})$ such that $\sum_{i=1}^n\Phi_i(I)=I$. Then, Corollary \ref{koro_merc} provides refinement of the Jensen-Mercer inequality for
 continuous non-negative superquadratic function $f:[0,\infty)\rightarrow \mathbb{R}$. To see this,  define $C_i=M+m-B_i$ $i=1,2,\ldots , n$. Obviously, $m\leq C_i\leq M$ and $B_i+C_i=M+m$. Now,
  setting  $A_i=mI$ and $D_i=MI$, $i=1,2,\ldots , n$, the second inequality in Corollary \ref{koro_merc} reduces to
{\footnotesize\begin{align*}
&f\left(M+m-\sum_{i=1}^{n}\Phi_i(B_i)\right)\\&\quad+
\frac{M-\sum_{i=1}^{n}\Phi_i(B_i)}{M-m}\left(\sum_{i=1}^{n}\Phi_i(f(B_i))-m\right)
+{\frac{\sum_{i=1}^{n}\Phi_i(B_i)-m}{M-m}}\left(M-\sum_{i=1}^{n}\Phi_i(f(B_i))\right)\\&\quad
 +\frac{\sum_{i=1}^{n}\Phi_i(B_i)-m}{M-m}f\left(M-\sum_{i=1}^{n}\Phi_i(B_i)\right)
 +{\frac{M-\sum_{i=1}^{n}\Phi_i(B_i)}{M-m}}f\left(\sum_{i=1}^{n}\Phi_i(B_i)-m\right)\\&\leq
f\left(m\right)+f(M)-\sum_{i=1}^{n}\Phi_i(f\left(B_i)\right)-2f(0),
\end{align*}}which represents more accurate form of the Jensen-Mercer inequality \eqref{jenmer} in the case of a non-negative function $f$.
\end{remark}

If $\Phi$ is an identity map, then the condition $A+D=B+C$ in Theorem \ref{main1} can be relaxed, as in Theorem \ref{main111} from the previous section. In fact, the following result is an analogue of
Theorem \ref{main111} for superquadratic functions.

\begin{theorem}\label{main11}
Let $f: [0,\infty)\rightarrow \mathbb{R}$ be  continuous superquadratic function and let $A, B, C, D\in{\mathbb B}({\mathscr H})_h$ be  such that  $0\leq A \leq m \leq B, C \leq M \leq D$, for  non-negative real numbers $0\leq m < M$.
If one of the following conditions
\begin{align*}
&({\rm i})\,\, B + C \leq A +D\quad \textrm{and}\quad f(m)\leq f (M)
\\&({\rm ii})\, A + D \leq B +C \quad\textrm{and}\quad f(M) \leq f (m)
\end{align*}
holds, then
\begin{align*}
f(B)&+f(C)+{\frac{B-m}{M-m}}f(M-B)
+\frac{M-C}{M-m}f(C-m)\\&\quad+{\frac{C-m}{M-m}}f(M-C)+\frac{M-B}{M-m}f(B-m)\\&
\leq f(A)+f(D)-f(m-A)-{\frac{m-A}{M-m}}f(M-m)\\&\quad-f(D-M)-{\frac{D-M}{M-m}}f(M-m).
\end{align*}
\end{theorem}
\begin{proof}
Since $A \leq m \leq B, C \leq M \leq D$,  utilizing the scalar inequalities derived in the proof of Theorem \ref{main1}, we have
\begin{align*}
f(B)&\leq \frac{f(M)-f(m)}{M-m}B+\frac{f(m)M-f(M)m}{M-m}\nonumber\\&\quad-\frac{M-B}{M-m}f(B-m)-{\frac{B-m}{M-m}}f(M-B)
\end{align*}
\begin{align*}
f(C)&\leq \frac{f(M)-f(m)}{M-m}C+\frac{f(m)M-f(M)m}{M-m}\nonumber\\&\quad-\frac{M-C}{M-m}f(C-m)-{\frac{C-m}{M-m}}f(M-C),
\end{align*}
\begin{align*}
f(A)&\geq \frac{f(M)-f(m)}{M-m}A+\frac{f(m)M-f(M)m}{M-m}\nonumber\\&\quad+f(m-A)+{\frac{m-A}{M-m}}f(M-m)
\end{align*}
and
\begin{align*}
f(D)&\geq \frac{f(M)-f(m)}{M-m}D+\frac{f(m)M-f(M)m}{M-m}\nonumber\\&\quad+f(D-M)+{\frac{M-D}{M-m}}f(M-m).
\end{align*}
Now, summing the estimates for $f(B)$ and $f(C)$, and taking into account that $\frac{f(M)-f(m)}{M-m}(B+C)\leq \frac{f(M)-f(m)}{M-m}(A+D)$, it follows that
\begin{align*}
f(B)&+f(C)\\&
\leq\frac{f(M)-f(m)}{M-m}(A+D)+2\frac{f(m)M-f(M)m}{M-m}-\frac{M-B}{M-m}f(B-m)\\&\quad-{\frac{B-m}{M-m}}f(M-B)
-\frac{M-C}{M-m}f(C-m)-{\frac{C-m}{M-m}}f(M-C).
\end{align*}
Finally, the result follows by taking into account estimates for $f(A)$ and $f(D)$.
\end{proof}

\begin{remark}
If $f: [0,\infty)\rightarrow \mathbb{R}$ is a continuous superquadratic function and  $(A,D)\in\Omega$, where $\Omega$ is defined in Remark \ref{midpoint1},
then Theorem \ref{main11} implies the inequality
\begin{align*}
&f\Big(\frac{A+D}{2}\Big)\\&\quad +{\frac{A+D-2m}{2(M-m)}}f\Big(M-\frac{A+D}{2}\Big)
+\frac{2M-A-D}{2(M-m)}f\Big(\frac{A+D}{2}-m\Big)\\
\leq &\quad \frac{1}{2}f(A)+\frac{1}{2}f(D)\\
&\quad-\frac{1}{2}f(m-A)-\frac{1}{2}f(D-M)-{\frac{m-A+D-M}{2(M-m)}}f(M-m).
\end{align*}
\end{remark}

\begin{remark}
It should be noticed here that some related  Jensen-type operator inequalities for superquadratic functions have been studied in \cite{bac} and \cite{kian}.
\end{remark}

%===================================================================================================================================


\begin{thebibliography}{99}


\bibitem{AJS} S. Abramovich, G. Jameson,  G. Sinnamon, \textit{Refining
of Jensen's inequality, }Bull. Math. Soc. Sci.\ Math. Roumanie
(N.S.) \textbf{47} (2004), 3--14.

\bibitem{AJS1} S. Abramovich, G. Jameson,  G. Sinnamon, \textit{%
Inequalities for averages of convex and superquadratic functions,
}J.
Inequal. Pure Appl. Math.\textit{\ } \textbf{7} (2004)\textbf{, }%
Art. 70.

\bibitem{bakh} M. Bakherad, M. Krni\'{c}, M.S. Moslehian, \textit{Reverses of the Young inequality for matrices and operators},
 Rocky Mountain J. Math. \textbf{46} (2016), 1089--1105.

\bibitem{bac} J. Bari\' c, A. Matkovi\' c, J. Pe\v cari\' c,  \textit{A variant of the Jensen-Mercer operator inequality for superquadratic functions},
 Math. Comput. Modelling \textbf{51} (2010), 1230--1239.

\bibitem{abc} T. Furuta, J. Mi\' ci\' c Hot, J. Pe\v cari\' c, Y. Seo, {\it
Mond-Pe\v cari\' c Method in Operator Inequalities}, Element,
Zagreb, 2005.






\bibitem{kian}  M. Kian, \textit{Operator Jensen inequality for superquadratic functions},
Linear Algebra Appl.   \textbf{456}    (2014),   82--87.

\bibitem{analele} M. Krni\'c, N. Lovri\v cevi\'c, J. Pe\v cari\'
c, {\it Jessen's functional, its properties and applications}, An.
\c{S}t. Univ. Ovidius Constan\c{t}a {\bf 20} (2012), 225--248.

\bibitem{mat}  A. Matkovi\'{c}, J. Pe\v{c}ari\'{c}, I. Peri\'{c}, \textit{A variant of Jensen's inequality of Mercer's type for operators with applications},
Linear Algebra Appl. \textbf{418} (2006), 551--564.


\bibitem{mos} M.S. Moslehian, J. Mi\'{c}i\'{c}, M. Kian, \textit{An operator inequality and its consequences},
Linear Algebra Appl. \textbf{439} (2013), 584--591.

\bibitem{mos2} M.S. Moslehian, J. Mi\'{c}i\'{c}, M. Kian, \textit{Operator inequalities of Jensen type}, Top. Algebra Appl. \textbf{1} (2013), 9--21.

\bibitem{sab1} M. Sababheh,  \textit{Means refinements via convexity}, Mediterr. J. Math.  (2017) 14:125. https://doi.org/10.1007/s00009-017-0924-8.

\end{thebibliography}
\end{document}